\documentclass[preprint]{elsarticle}

\usepackage{mathptm}    
\usepackage{latexsym}   
\usepackage{amssymb}    
\usepackage{amsmath}    
\usepackage{amsthm}     
\usepackage{url}

\newtheorem{theorem}{Theorem}[section]
\newtheorem{lemma}[theorem]{Lemma}

\newtheorem{corollary}[theorem]{Corollary}

\newenvironment{prf}[1]{\trivlist
\item[\hskip
\labelsep{\it #1.\hspace*{.3em}}]}{
\endtrivlist}

\newtheorem{predefinition}[theorem]{Definition}

\newtheorem{preremark}[theorem]{Remark}

\newtheorem{prenotation}[theorem]{Notation}

\newtheorem{preexample}[theorem]{Example}
\newenvironment{example}{\begin{preexample}\rm}{\end{preexample}}
\newtheorem{preclaim}[theorem]{Claim}

\newtheorem{prequestion}[theorem]{Question}

 \makeatletter
\def\emppsubsection{\@startsection{subsection}{2}{\z@}{-3.25ex plus -1ex minus -.2ex}{-1em}}

\makeatother

\def\FF{{\mathbb F}}

\input diagxy

\begin{document}

\title{Pointless Hyperelliptic Curves}
\author{Ryan Becker}
\ead{beckry01@gettysburg.edu}
\address{Department of Mathematics, Gettysburg College, 200 N. Washington Street, Gettysburg PA 17325}

\author{Darren Glass}
\ead{dglass@gettysburg.edu}
\address{Department of Mathematics, Gettysburg College, 200 N. Washington Street, Gettysburg PA 17325}


\begin{abstract}

In this paper we consider the question of whether there exists a hyperelliptic curve of genus $g$ which is defined over $\FF_q$ but has no rational points over $\FF_q$ for various pairs $(g,q)$.

\end{abstract}

\maketitle

\section{Introduction and Background}

The question of constructing curves of a given genus with a given number of points over a finite field is one that many mathematicians have worked on.  Because of applications in coding theory, most of the energy has been spent trying to find curves of a fixed genus with as many points as possible -- much of this work is described at the website \url{http://www.manYPoints.org} \cite{MP}, maintained by Gerard van der Geer and others.  Interested readers may wish to consult \cite{FG}, \cite{GK}, \cite{HKT}, \cite{L2} among other papers.  While it may lack immediate applications, it is nonetheless an interesting mathematical question to consider how few points a curve of a given genus might have, and in particular whether there exist curves of a given genus that do not have any points defined over a fixed finite field.  For small genera, this question was considered in \cite{HLT} and \cite{Stark}, where the authors  proved the following results.

\begin{theorem}\label{T:smallgenus}
The following conditions on the existence of pointless curves are both necessary and sufficient:
\begin{itemize}
\item There exist pointless curves of genus $2$ defined over $\FF_q$ if and only if $q<13$.
\item There exist pointless curves of genus $3$ defined over $\FF_q$ if and only if $q \le 25, q=29$ or $q=32$.
\item There exist pointless hyperelliptic curves of genus $3$ defined over $\FF_q$ if and only if $q \le 25$.
\item There exist pointless curves of genus $4$ defined over $\FF_q$ if and only if $q \le 49$.
\end{itemize}
\end{theorem}

The results of Theorem \ref{T:smallgenus} fix a genus and let the field vary; in this note, we take the complementary point of view and fix our finite field and consider for which genera there exists a pointless curve.  A recent result of Stichtenoth in \cite{Sti} proved that for each finite field $\FF_q$, there exists a number $g_q$ so that for all $g \ge g_q$ there is a pointless curve over $\FF_q$ of genus $g$.  In this note, we consider the analogous question for hyperelliptic curves.  In particular, Section \ref{S:bounds} gives two types of explicit constructions of pointless hyperelliptic curves of various genera which give bounds on $g_q$.  One of these relies on the value of $g$ mod $q$ and the other on the value of $g$ mod $q-1$.  Examples of such results are:

\begin{corollary}\label{C:Bound}
Let $a$ be the least residue of $g$ mod $p$.  Then there exist pointless hyperelliptic curves of genus $g$ over $\FF_p$ if $g \ge (p-a)(p-2)$.  In particular, there will be pointless hyperelliptic curves of every genus $g \ge  \frac {(p+1)(p-2)}{2}$.
\end{corollary}

\begin{corollary}\label{C:relprime}
Let $(g+1,\frac{p-1}{2}) = 1$ and $g\ge \frac{p-1}{2}$.  Then there is a pointless hyperelliptic curve of genus $g$ over $\FF_p$.
\end{corollary}

In Section \ref{S:numerics}, we combine these results as well as some results about fibre products of hyperelliptic curves to give explicit numerical bounds for when pointless hyperelliptic curves can exist over a specific finite field.  These results show that in general there is a linear bound on $q$ above which one can obtain all genera.  We note that Serre's bound says that there is a lower bound on the order of $\sqrt{q}$ below which we cannot obtain any pointless hyperelliptic curves.  In future work, we hope to explore the gap between these two bounds.

\section{Existence Results} \label{S:bounds}

Before we prove our main results, we begin by introducing some notation.  Let $\FF_q$ be the field of odd order $q=p^r$ and let $C$ be a hyperelliptic curve of genus $g$ defined over $\FF_q$ by the equation $y^2=f(x)$.  Let $n$ be the number of points of $C$ defined over $\FF_q$.  If we choose $a \in \FF_q$ to be a nonsquare and define $\tilde{C}$ to be the quadratic twist of $C$ given by $y^2=af(x)$, then $\tilde{C}$ will be a hyperelliptic curve of genus $g$ with $2q+2-n$ points defined over $\FF_q$.  In particular, there will be pointless hyperelliptic curves over $\FF_q$ if and only if there are curves with $2q+2$ points over $\FF_q$. It is often convenient to consider these curves, which have the maximal number of points allowable for hyperelliptic curves, instead of curves with no points.

As an example of this approach, consider the curve $C_g$, defined by the equation \[y^2 = f(x) = x^{2g+2}-x^{2g+2-q+1}+1\]  where $2g+2 \ge q$.  It is clear that for any $x \in \FF_q$, $f(x) = 1$ and therefore there will be two points lying over each $x$ value.  Moreover, $f(x)$ is monic of even degree, so there are also two points lying over $x = \infty$.  It follows that $C_g$ has $2q+2$ points over $\FF_q$ and that the quadratic twist $\tilde{C_g}$ is pointless.  In order to show that the curves $C_g$ and $\tilde{C_g}$ have genus $g$, it suffices to show that they are nonsingular. In particular, we wish to show that $f(x)$ has no repeated roots over $\overline{\FF_q}$.  It will follow that $\tilde{C_g}$ is a pointless hyperelliptic curve of genus $g$.

\begin{theorem}\label{T:modp}
Let $a$ be the least residue of $g$ mod $p$.  There exists a $2p+2$-pointed curve of genus $g$ defined over $\FF_q$ if $g \ge (p-a-1)(q-1)$. If $0 \le a \le \frac{p-3}{2}$, then there exists a $2p+2$-pointed curve of genus $g$ if $g \ge \frac{q-1}{2}(p-2a-2)$.
\end{theorem}

\begin{proof}
Let $l$ be an integer congruent to $-(2g+2)$ mod $p$ so that $2g+2>l(q-1)>0$. Consider the equation $f(x)=x^{2g+2}-x^{2g+2-l(q-1)}+1$.  We show that the curve defined over $\FF_q$ by the equation $y^2=f(x)$ is nonsingular by showing that $f(x)$ has no roots in common with its derivative.  We make the following computations in $\FF_q$:

\begin{eqnarray*}
f'(x) &=& (2g+2)x^{2g+1}-(2g+2-l(p-1))x^{2g+1-l(p-1)} \\
&=&x^{2g+1-l(p-1)}((2g+2)x^{l(p-1)}-(2g+2)+lp-l) \\
&=& (2g+2)x^{2g+1}
\end{eqnarray*}

It follows that all roots of $f'(x)$ are 0, which is not a root of $f(x)$, so in particular $f(x)$ and $f'(x)$ do not share any roots.  This implies that the curve defined by $y^2=f(x)$ is a nonsingular hyperelliptic curve of genus $g$.  Moreover, for each $x \in \FF_q$ it is clear that $f(x) = 1$ and therefore there are two choices of $y$ which correspond to this $x$, and the fact that $f(x)$ is monic implies that there are two points over $\infty$.  Thus, the curve has $2q+2$ points.

The theorem will now follow from a change of variables.  In particular, the sufficient condition is the existence of a positive integer $l$ such that $2g+2 \equiv l(p-1)$ mod $p$ and $2g+2 > l(q-1)$.  Setting $a$ to be the least residue of $g$ mod $p$, the first condition is equivalent to saying that $l \equiv -(2a+2)$ mod $p$.

If $a=p-1$ then the smallest positive choice of $l$ satisfying this condition will be $l = p$.  The second condition then requires that $g \ge p(\frac{q-1}{2})$.

If $0\le a \le p-1$ then we have that $2 \le (2a+2) \le 2p$ and we can choose $l=-(2a+2)+2p$ as long as $2g+2 >(-(2a+2)+2p)(q-1)$.  This latter condition is equivalent to:
\[g \ge (p-a-1)(q-1)\]

If $a \le \frac{p-3}{2}$, then we can improve this bound by instead setting $l=-(2a+2)+p$. In this case, we find the stronger bound that
\begin{eqnarray*}
2g+2 &>& (-(2a+2)+p)(q-1)\\
g &\ge & \frac{q-1}{2}(p-2a-2)
\end{eqnarray*}
This proves the stated theorem.
\end{proof}

In the case where we are working over a prime field, the above result simplifies to Corollary \ref{C:Bound}.  We get a different type of result by considering the congruence class of $g$ mod $q-1$ rather than mod $q$.

\begin{lemma}\label{T:modp'}
Assume that $2g+2 \ge q$ and set $j$ to be the least residue of $2g+2$ mod $q-1$. Furthermore, let $d = gcd(j,q-1)$ and $k=\frac{2g+2-j}{q-1}$.  Then the curve defined by the equation $y^2= x^{2g+2}-x^{(2g+2)-(q-1)}+1$  is a nonsingular curve of genus $g$ unless $(k-j-1)^d \equiv (k-j)^d$ (mod $p$)
\end{lemma}

\begin{proof}
In order to show that this curve is nonsingular of genus $g$, it will suffice to show that the equation $f(x) = x^{2g+2}-x^{2g+2-q+1}+1$ has no double roots over $\overline{\FF_p}$.  We proceed by contradiction and assume that $\gamma \in \overline{\FF_p}$ is a double root of $f(x)$.  Then $f'(\gamma) = (2g+2) \gamma^{2g+1} - (2g+2-q+1) \gamma^{2g+2-q} = 0$.  Our hypotheses imply that $2g+2 \equiv j-k$ mod $q$, and therefore the following statement holds in $\FF_q$:
\[(j-k) \gamma^{2g+1} = (j-k+1) \gamma^{2g+2-q}\]
Note that $\gamma=0$ is not a root of $f(x)$, so we must have $\gamma^{q-1} = \frac{j-k+1}{j-k}$.  Plugging this into our formula for $f(x)$, we now see that
\[0 = f(\gamma) =  \left(\frac{j-k+1}{j-k}\right)^k \gamma^j - \left(\frac{j-k+1}{j-k}\right)^{k-1} \gamma^j + 1\]
This in particular allows us to deduce that $\gamma^j \in \FF_q$ and therefore $(\gamma^{q-1})^j =1$.  This in particular implies that $(\gamma^{q-1})^d = 1$.  The lemma is an immediate consequence.
\end{proof}

Looking at the quadratic twist, the following result is an immediate corollary:

\begin{corollary}
There exists a pointless hyperelliptic curve of genus $g$ defined over $\FF_q$ if $(k-j-1)^d \not \equiv (k-j)^d$ (mod $p$), where $j,k,d$ are as defined in the statement of Lemma \ref{T:modp'}.
\end{corollary}

\begin{example}
As an example of Lemma \ref{T:modp'}, we consider curves over $\FF_q$ of genus $g=q-4$.  In this case, $j=q-5$, so $d=(j,q-1)$ will be equal to $4$ (resp. $d=2$) if $q \equiv 1$ (resp. $q \equiv 3$) (mod $4$).  In particular, Lemma \ref{T:modp'} implies that the curve defined by $y^2 = x^{2q-6}-x^{q-5} + 1$ will be nonsingular except possibly in cases where $5^4=6^4$.  This implies the existence of a pointless hyperelliptic curve of genus $q-4$ unless $p=11$ or $p=61$.
\end{example}

A similar argument will show that there exist pointless hyperelliptic curves defined over $\FF_q$ of genus $g=q-a$ as long as $g \ge \frac{q-1}{2}$ and $p \not | ((2a-2)^{2a-4}-(2a-3)^{2a-4})$.  This gives an explicitly computable finite set of characteristics away from which we will have pointless hyperelliptic curves of a given genus. This approach generalizes, and while stating a result in full generality is difficult, we give one example below of such a result over $\FF_p$.  Note that $2g+2$ and $p-1$ are both even, so at best we have that $(2g+2,p-1)=2$, which is equivalent to the condition that $g+1$ and $\frac{p-1}{2}$ are relatively prime.  For notational convenience, we set $p'=\frac{p-1}{2}$.

\begin{theorem}\label{T:rr}
Assume that $(g+1,p')=1$.  Then the equation $x^{2g+2}-x^{2g+3-p}+1$ has repeated roots over $\overline{\FF_p}$ if and only if $g \equiv \frac{p^2-5}{4}+\frac{jp}{2}$ mod $pp'$ and one of the following cases hold:

\begin{itemize}
\item $p \equiv 3$ mod $8$ and $j \equiv 2k$ mod $4$
\item $p \equiv 7$ mod $8$ and $j \equiv 2k+2$ mod $4$
\item $p \equiv 5$ mod $8$ and $j \equiv 2$ mod $4$
\end{itemize}
\end{theorem}

\begin{proof}
The hypotheses imply that $gcd(2g+2,p-1) = 2$.  It follows from Lemma \ref{T:modp'} that $f(x)$ has no repeated roots unless $\left(\frac{k-j-1}{k-j}\right)^2 \equiv 1$ mod $p$, where $j$ is as in the statement of Lemma \ref{T:modp'}.  Clearly $k-j-1 \not \equiv k-j$, so it suffices to consider the case where $k-j-1 \equiv -(k-j)$.  Computing mod $p$:

\begin{eqnarray*}
k-j-1 & \equiv & -(k-j) \\
2(j-k) &\equiv &-1 \\
2(2g+2) & \equiv & -1 \\
4g & \equiv & -5\\
g & \equiv & \frac{p^2-5}{4} \\
\end{eqnarray*}

This last congruence follows from the fact that the multiplicative inverse of $4$ mod $p$ can always be expressed as $\frac{1-p^2}{4}$. Recall that we also know that $g \equiv \frac{j-2}{2}$ mod $p'$ and therefore it follows from the Chinese Remainder Theorem that $g \equiv \frac{p^2-5}{4}+p\cdot \frac{j}{2}$ mod $pp'$.

Furthermore, we have shown above that $\gamma$ will be a double root of $f(x)$ if and only if it is a root such that $\gamma^{p-1} \equiv -1$ mod $p$.  In particular, this implies that $\gamma^j \equiv (-1)^k p'$.  Alternatively, there will be a double root if and only if there exists an element $\alpha \in \FF_p$ such that $\alpha$ is a quadratic nonresidue such that $\alpha^{j/2} \equiv (-1)^k p'$.

Assume $j/2$ is odd.  Then in particular $gcd(j/2,p-1)=1$ and therefore there is a unique $\alpha$ satisfying the equation $\alpha^{j/2} \equiv (-1)^k p'$.  Moreover, it will be a nonresidue exactly when $(-1)^k p'$ is.

Next, assume $j/2$ is even.  In  particular, we note that this implies that $p \equiv 3$ mod $4$.  If $(-1)^k p'$ is a quadratic nonresidue then there are no $\alpha$ which satisfy the above equation.  On the other hand, if it is a quadratic residue then there will be two choices of $\beta$ such that $\beta^2 \equiv (-1)^k p'$ and, because $-1$ is a nonresidue mod $p$ it will follow that exactly one of these two choices will itself be a quadratic residue.  Assume $j/4$ is odd.  Then it follows from the fact that $gcd(j/4,p-1) = 1$ that there is a unique choice of $\alpha$ so that $\alpha^{\frac{j}{4}} \equiv \beta$ and by choosing $\beta$ to be the nonresidue we will get a unique nonresidue $\alpha$ such that $\alpha^{\frac{j}{2}} \equiv \beta^2 \equiv(-1)^k p'$.  In particular, the curve will be singular when $(-1)^k p'$ is a quadratic residue. If $j/4$ is even then we can repeat the process, and a similar argument will show that anytime $(-1)^k p'$ is a quadratic residue we will find a unique nonresidue $\alpha$ with $\alpha^{\frac{j}{2}} = (-1)^k p'$.

In particular, we have shown that if $j \equiv 2$ mod $4$ then the curve will be singular if $(-1)^k p'$ is a nonresidue and if $j \equiv 0$ mod $4$ it will be singular if $(-1)^k p'$  is a residue.  The theorem follows.
\end{proof}

We conclude this section with a proof of Corollary \ref{C:relprime}.

\begin{proof}
From the above Theorem, it follows that curve defined by $y^2=x^{2g+2}-x^{2g+2-(p-1)}+1$ is nonsingular unless $g \equiv \frac{p^2-5}{4}+p\cdot \frac{j}{2}$ mod $pp'$.  Note that if $j<p'$ then the quantity on the right hand side is less than $pp'$ and if $p' < j < p-1$ then it is less than $2pp'$. (Note that $j \ne p'$ by our assumption on the $gcd(j,p-1)$.)

We now turn to Theorem \ref{T:modp}, which tells us that if $2g+2 \equiv p'$ mod $p$ then we can construct pointless hyperelliptic curves as long as $g> \frac{p^2-5}{4}$.

In particular, the only cases not covered by the theorems above are when $j$ is an even integer relatively prime to $p-1$ in the interval $(p',p-1)$ and $g = \frac{2p+2pj-p^2-5}{4}$.  In each of these cases, the equation $x^{2g+2}-x^{2g+2-(p-1)}+1$ has repeated roots by Theorem \ref{T:rr} but the genus is too small to consider curves of the form $x^{2g+2}-x^{2g+2-\ell(p-1)}+1$ for $\ell>1$.

\end{proof}

\section{Numerical Results} \label{S:numerics}

In this section, we construct tables of genera unobtainable using the above theorems for all odd primes less than 100. For each prime $p$, we begin by checking, for all genera less than the bound established by Corollary \ref{C:Bound}, whether the conditions of Theorem \ref{T:modp} are satisfied.  We thereby obtain a set of genera unobtainable using that result alone.  We further prune this set using the conditions of Theorem \ref{T:modp'} and of Corollary \ref{C:relprime}. Another result that will be useful in our computations is the following theorem.

\begin{theorem}\label{T:double}
Let $C$ be a pointless hyperelliptic curve of genus $g$ defined by the equation $y^2=f(x)$.  Then the equation $y^2=f(x^2)$ defines a pointless hyperelliptic curve of genus $2g+1$.
\end{theorem}

One can prove this theorem in several manners.  We give a proof involving fibre products as the technique will be useful later on.

\begin{proof}
The fact that $C$ has no points defined over $\FF_q$ implies in particular that $f(x)$ has degree $2g+2$ and that $f(0) \ne 0$.  In particular, neither $0$ nor $\infty$ is a ramification point of $f(x)$.  We wish to consider the normalization of the fibre product of $C$ with the hyperelliptic cover given by $z^2=x$ which is branched only at the points $x=0,\infty$.  It follows from results about fibre products of hyperelliptic curves (see \cite{GP}, \cite{GV} for details and similar constructions) that this curve will have genus $2g+1$ and will have no points defined over $\FF_q$.  This curve is equivalent to the one defined by $y^2=f(x^2)$ after a change of variables.
\end{proof}

The table in Figure 1 gives a complete list of genera that our methods are unable to construct (`missed genera') for $p \le 23$.  This list gets lengthy as $p$ gets larger, so for $p>23$ we instead present the number of missed genera and the largest missed genus.

\begin{figure}[h]
\begin{center}
\begin{tabular}{|c|p{10cm}|}
\hline
$p$ & Missed Genera \\ \hline
3 & $\emptyset$ \\ \hline
5 & 3, 7 \\ \hline
7 & 2, 5, 8, 11, 14, 17 \\ \hline
11 & 2, 3, 4, 9, 14, 19, 24, 29, 34, 39, 44, 49 \\ \hline
13 & 2, 3, 4, 5, 8, 11, 17, 20, 23, 35, 41, 47, 71 \\ \hline
17 & 2, 3, 4, 5, 6, 7, 15, 31, 63, 127 \\ \hline
19 & 2, 3, 4, 5, 6, 7, 8, 14, 17, 26, 32, 35, 38, 44, 50, 53, 62, 71, 80, 89, 98, 107, 116, 125, 134, 143, 152, 161 \\ \hline
23 & 2, 3, 4, 5, 6, 7, 8, 9, 10, 21, 32, 43, 54, 65, 76, 87, 98, 109, 120, 131, 142, 153, 164, 175, 186, 197, 208, 219, 230, 241 \\ \hline
\end{tabular}
\caption{All Missed Genera For Small Primes}
\end{center}
\end{figure}

\begin{figure}[h]
\begin{center}
\begin{tabular}{|c|c|c|}
\hline
$p$ & Number Missed & Largest Missed \\ \hline
29 & 31 & 391 \\ \hline
31 & 52 & 449 \\ \hline
37 & 52 & 647 \\ \hline
41 & 35 & 759 \\ \hline
43 & 76 & 881 \\ \hline
47 & 66 & 1057 \\ \hline
53 & 67 & 1351 \\ \hline
59 & 84 & 1681 \\ \hline
61 & 95 & 1799 \\ \hline
67 & 121 & 2177 \\ \hline
71 & 133 & 2449 \\ \hline
73 & 86 & 2483 \\ \hline
79 & 143 & 3041 \\ \hline
83 & 120 & 3361 \\ \hline
89 & 82 & 3079 \\ \hline
97 & 83 & 4607 \\ \hline
\end{tabular}
\caption{Data on Genera Missed For $p <100$}
\end{center}
\end{figure}

We note that only results proven in this note are used even when the tables can be improved by other results. For example, from Theorem \ref{T:smallgenus}, we know that there exists a pointless hyperelliptic curve of genus 3 over $\FF_{17}$.  Theorem \ref{T:double} then allows us to obtain pointless curves of genera 7, 15, 31, 63, and 127.  One can similarly obtain select other genera listed in the above chart.

Moreover, these numerical results do not take into account how $f(x)$ factors over $\FF_p$, as these computations are beyond the scope of the program used. However, one can prove the following result.

\begin{theorem}\label{T:factor}
Let $C$ be a pointless hyperelliptic curve of genus $g$ defined by the equation $y^2=f(x)$.
\begin{itemize}
\item If $f(x)$ has any factor over $\FF_p$ that is not given by an irreducible quadratic equation then there exists a pointless hyperelliptic curve of genus $2g$.
\item If $f(x)$ has a factor defined by an irreducible quadratic equation over $\FF_p$ then there exists a pointless hyperelliptic curve of genus $2g-1$.
\end{itemize}
\end{theorem}

The proof of Theorem \ref{T:factor} is similar to the proof of Theorem \ref{T:double}.  However, instead of taking a fibre product with a cover defined by a quadratic equation sharing no roots with $f(x)$ we instead take the fibre product with a cover defined by a quadratic equation which shares either one or two roots with $f(x)$, whose existence is guaranteed by the hypothesis.  We omit the details of the proof.

\bibliographystyle{amsplain}
\bibliography{pointless}

\end{document}